\pdfoutput=1
\documentclass{amsart}
\usepackage[T1]{fontenc}
\usepackage{graphicx}
\usepackage{amsthm}
\usepackage{amsmath}
\usepackage{amssymb}
\usepackage{hyperref}

\newcommand{\tr}{
\begin{picture}(10,0)
\put(5,3){\circle{9}}
\put(5,-1.5){\line(0,1){9}}
\qbezier(2,0)(5,3)(8,6)
\qbezier(2,6)(5,3)(8,0)
\end{picture}
}
\newcommand{\x}{
\begin{picture}(10,0)
\put(5,3){\circle{9}}
\qbezier(2,0)(5,3)(8,6)
\qbezier(2,6)(5,3)(8,0)
\end{picture}
}
\newcommand{\h}{
\begin{picture}(10,0)
\put(5,3){\circle{9}}
\put(0,3){\line(1,0){9}}
\qbezier(2,0)(2,3)(2,6)
\qbezier(8,0)(8,3)(8,6)
\end{picture}
}

\theoremstyle{definition}
\newtheorem{theorem}{Theorem}[section]
\newtheorem{corollary}[theorem]{Corollary}
\newtheorem{lemma}[theorem]{Lemma}
\newtheorem{habiro}{Habiro's Theorem}

\newtheorem{formulation}{Corollary to  Habiro's Theorem}

\newtheorem{conjecture}[theorem]{Conjecture}
\newtheorem{definition}[theorem]{Definition}
\newtheorem{notation}[theorem]{Notation}
\newtheorem{remark}[theorem]{Remark}
\newtheorem{fact}[theorem]{Fact}

\author{Noboru Ito}
\address{Graduate School of Mathematical Sciences, The University of Tokyo, 3-8-1, Komaba, Meguro-ku, Tokyo, 153-8914, Japan}
\address{Current address: National Institute of Technology, Ibaraki College, 866, Nakane, Hitachinaka, Ibaraki, 312-8508, Japan}
\email{
nito@gm.ibaraki-ct.ac.jp
}
\thanks{The part of the work of N.~Ito was supported by JSPS Japanese-German Graduate Externship, a Waseda University Grant for Special Research Projects (Project number: 2015K-342), and MEXT KAKENHI Grant Number 20K03604.  
N.~Ito was a project researcher of Grant-in-Aid for Scientific Research (S) 24224002. }
\author{Yusuke Takimura}
\address{Gakushuin Boys' Junior High School, 1-5-1, Mejiro, Toshima-ku, Tokyo, 171-0031, Japan}
\email{Yusuke.Takimura@gakushuin.ac.jp}
\date{July 25, 2021}
\begin{document}
\noindent                                             
\begin{picture}(150,36)                               
\put(5,20){\tiny{Submitted to}}                       
\put(5,7){\textbf{Topology Proceedings}}              
\put(0,0){\framebox(140,34){}}                        
\put(2,2){\framebox(136,30){}}                        
\end{picture}                                        

\renewcommand{\bf}{\bfseries}
\renewcommand{\sc}{\scshape}
\vspace{0.5in}
\begin{abstract}
A generic immersion of a circle into a $2$-sphere is often studied as a projection of a knot; it is called a knot projection.   
A chord diagram is a configuration of paired points on a circle; traditionally, the two points of each pair are connected by a chord.  A triple chord is a chord diagram consisting of three chords, each of which intersects the other chords.   
Every knot projection obtains a chord diagram in which every pair of points corresponds to the inverse image of a double point.   
In this paper, we show that  for any knot projection $P$, if its chord diagram contains no triple chord, then there exists a finite sequence from $P$ to a simple closed curve such that the sequence consists of flat Reidemeister moves, each of which decreases $1$-gons or strong $2$-gons, where a strong $2$-gon is a $2$-gon oriented by an orientation of $P$.  
\end{abstract}
\title[Nontrivial knot projection with no triple chords]{Any nontrivial knot projection with no triple chords has a monogon or a bigon}
\maketitle
\section{\bf Introduction}
V.~I.~Arnold \cite{Arnold1994} (V.~A.~Vassiliev \cite{Vassiliev1990},~resp.) introduces a theory classifying plane curves (knots,~resp.) in vector spaces generated by immersions $S^1 \to \mathbb{R}^2$ ($\mathbb{R}^3$,~resp.) divided by subspaces, called discriminants, each of which consists of curves (knots,~resp.) with singularities.  For the knot case, it is well known that every coefficient in the Taylor expansion $t=e^x$ of the Jones polynomial is a Vassiliev invariant \cite{BirmanLin1993}.     
For plane curves, Arnold \cite[Page~16, Remark]{Arnold1994} remarked: 

\begin{quote}
{One may ask whether the series
\begin{equation}\label{ArnoldEq}
0 = J^+ + 2 St - E_1 + E_2 \cdots
\end{equation}
can be continued by adding natural invariants $E_i$ vanishing on increasing sets of ``simplest'' curves.
} 
\end{quote}


Here, $J^+ + 2 St$ is known as the Arnold invariant of spherical curves (including plane curves) and $\frac{1}{8} (J^+ + 2 St)$ is the average of the second coefficients of the Conway polynomials of the possible $2^n$ knots by over/under informations for spherical curves, i.e. knot projections.  
Arnold was interested in the filtered spaces that vanish on filtered (Vassiliev-type) invariants for plane curves \cite[Page~16, Remark]{Arnold1994}.  More precisely, we formulate (\ref{ArnoldEq}) as   Conjecture~\ref{conj:FilteredPlane}.  
\begin{conjecture}[Arnold \cite{Arnold1994}]\label{conj:FilteredPlane}
{\it
Let $J^+ + 2 St$ be as above.    
There exists a sequence of invariants $\{ E_i \}_{i \ge 0}$ with $E_0$ $=$ $J^+ +$ $2 St$ such that  $C$ is an immersion $S^1 \to S^{2}$ and 
\begin{equation}\label{1_eq1}
\{C~|~ E_0 (C)=0 \} \supset \{C~|~ E_1 (C)=0 \} \supset \{C~|~ E_2 (C)=0 \} \supset \cdots
\end{equation} 
}
\end{conjecture}
For knots and Vassiliev theory, 
(\ref{ArnoldEq}), called the {\it Arnold-Vassiliev approach} \cite{ParasolovSossinsky1997}, 
is very successful because Kohno \cite{Kohno1993} showed that Vassiliev invariants classify pure braids, and    
the new notion $C_k$-move to approach a knot classification was introduced by Habiro \cite{Habiro2000} (Goussarov \cite{Goussarov1999, Gusarov2000} gave similar results independently).  For the details and the definition of $C_k$-move, see \cite{Habiro2000}.    
\begin{habiro}[Habiro \cite{Habiro2000}]
{\it
Let $K$ be a knot.  The formula  
$v_m (K) = 0$ holds for any Vassiliev invariant $v_m$ $(m \le n)$ if and only if there exists a sequence of $C_{n+1}$-moves from $K$ to the unknot.    
} 
\end{habiro}     
It is known that $C_{m}$-moves generate $C_{m+1}$-moves \cite{Habiro2000} and $v_1 (K)=0$ for every $K$; Habiro's theorem implies a classification via  the sequence $\{K~|~ v_m (K)=0 \}$ ($m \ge 1$).   
\begin{formulation}
{\it 
Let $\mathcal{K}$ be the set of knots.  
There exists a sequence of Vassiliev knot  invariants $\{ v_i \}_{i \ge 1}$ such that  
\[ \mathcal{K} = 
\{K~|~ v_1 (K)=0 \} \supset \{ K~|~ v_2 (K)=0 \} \supset \{K~|~ v_3 (K)=0 \} \supset \cdots
\]
}
\end{formulation}
However, for plane/spherical curves, there have been unsolved problems even if they are fundamental in this approach which is similar to Vassiliev theory.  For example, we naturally list 
unsolved problems 
 when we compare Arnold plane curve theory with Vassiliev knot theory.      
\begin{itemize}
\item Find topological invariants ``$E_i$ ($i \in \mathbb{N}$)'' with filtered vector spaces.   
\item Find a global topological  property of this classification by filtered spaces.   
\item Find discriminants with degrees.  
\end{itemize}
Thus, as a first step, it is natural to consider sequence (\ref{1_eq1}) as in  Conjecture~\ref{conj:FilteredPlane}.      
However, even the first set $\{ C~|~ (J^+ + 2 St)(C) = 0 \}$ is still unknown.  

In this paper, in order to study Conjecture~\ref{conj:FilteredPlane},  we introduce sequence (\ref{1_eq2}) as in    Corollary~\ref{main1},   
which is given by  Theorem~\ref{main_thm}.     

\vspace{0.2in}

This paper is organized as follows.  In Section~\ref{sec2}, definitions are obtained and main results (Theorem~\ref{main_thm} and Corollary~\ref{main1}) are introduced.  In Section~\ref{sec3}, a proof of Theorem~\ref{main_thm} is obtained.      
An application of Theorem~\ref{main_thm}, including a formulation by Kouki Taniyama (personal communication, February 15, 2014) 
Corollary~\ref{cor2}, is presented in Section~\ref{sec5}.  

\section{\bf Definitions and the main result}\label{sec2}
\begin{definition}[knot projection, lune region]
A \emph{knot projection} is the image of a generic immersion of a circle into a $2$-sphere.   A self-intersection of a knot projection is a transverse double point which is simply referred to as a \emph{double point} in this paper.   
For a knot projection, if we find a disk with the boundary that consists of two double points and two arcs where some sub-curves may pass through the disk, then we call such a disk a \emph{lune region} (Figure~\ref{21}).  
In this paper, the knot projection that has no double points is called the \emph{simple closed curve} and is denoted by $U$.  
\begin{figure}[h!]
\includegraphics[width=3cm]{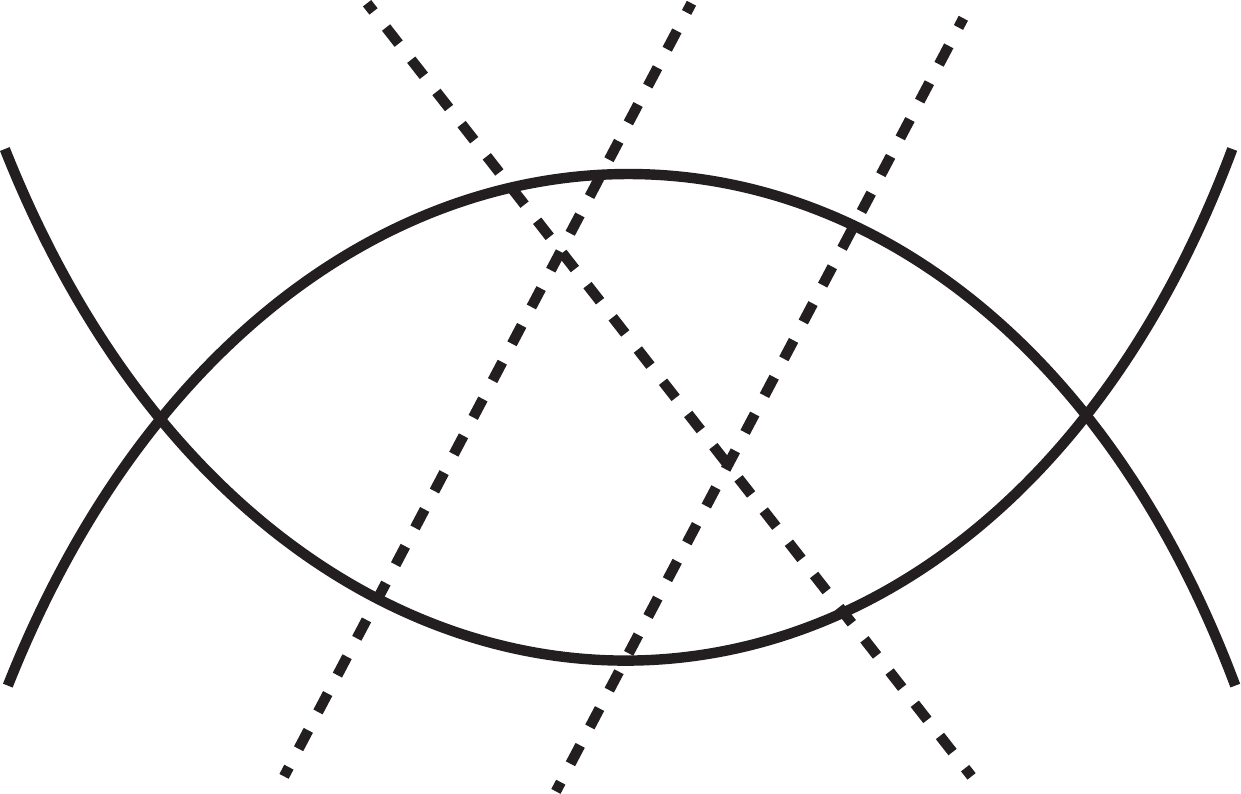}
\caption{A lune region.  The dotted curves indicate a possible parts of a given knot projection.}\label{21}
\end{figure}
\end{definition}
\begin{definition}[chord diagram]
A \emph{chord diagram} is a configuration of paired points on a circle up to ambient isotopy and reflection of the circle.      
Two points of each pair are usually connected by a simple arc, which is called a \emph{chord}.   
Every knot projection defines a chord diagram as follows (e.g., Figure~\ref{f1}).  
\begin{figure}[h!]
\includegraphics[width=8cm]{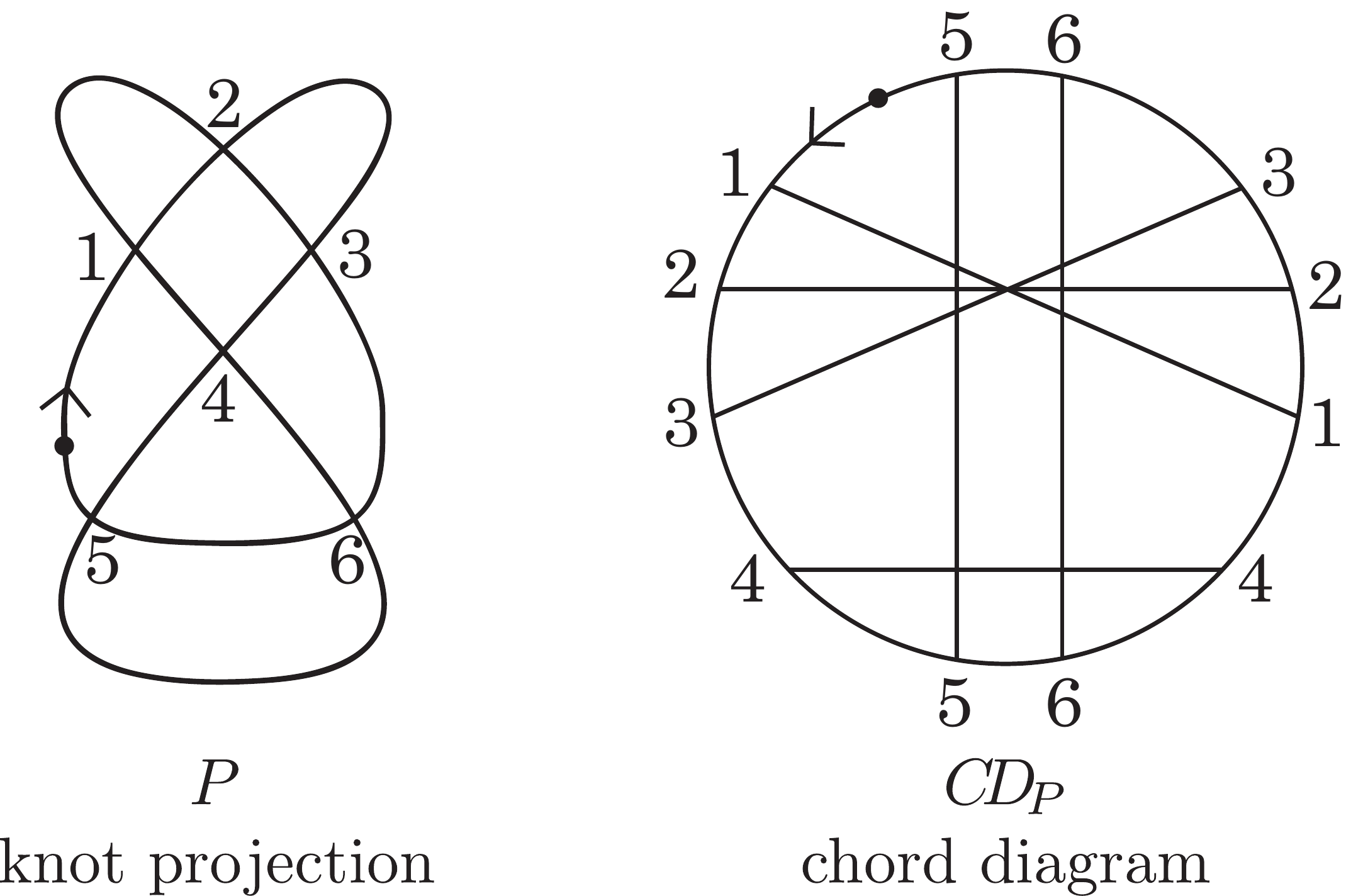}
\caption{A knot projection $P$ with a base point and an orientation (left), and a chord diagram $CD_P$ (right).}\label{f1}
\end{figure}
For a knot projection $P$, there exists a generic immersion $f :$ $S^{1} \to S^{2}$ such that $f(S^{1})=P$.   A chord diagram $CD_P$ of the knot projection $P$ is a circle with the preimage of each double point connected with a chord \cite{Ito2016}.  
\end{definition}
\begin{notation}[\x, \tr]\label{notation2}
Let $\{ p^{0}_1, p^{1}_1 \}$, $\{ p^{0}_2, p^{1}_2 \}, \ldots, \{ p^{0}_n, p^{1}_n \}$ be  paired points on a circle presenting a chord diagram.  (If the reader would like to  consider a general case, we recommend defining a chord diagram by an equivalence class of Gauss words \cite{FHIKM2018}).  
For example, when we go around the circle of a chord diagram, if four points ($p^{0}_1$, $p^{0}_2$, $p^{1}_1$, and $p^{1}_2$) appear in this order, then it is denoted by $\x$.  Similarly, when we go around the circle of a chord diagram, if six points ($p^{0}_1$, $p^{0}_2$, $p^{0}_3$, $p^{1}_1$, $p^{1}_2$, and  $p^{1}_3$) appear in this order, then it is denoted by $\tr$.      
For a knot projection $P$, the number of the sub-chord diagrams of type $\x$ ($\tr$,~resp.) embedded in $CD_P$ is denoted by $\x(P)$ ($\tr(P)$,~resp.).  In general, for a given chord diagram $x$, the function $x(P)$ is the number of the sub-chord diagrams of type $x$ embedded in $CD_P$.   (For the more formal definition of $x(P)$ obtained by Gauss words, see \cite{FHIKM2018}.)    

If $\tr(P) \neq 0$, we say that \emph{$CD_P$ contains a triple chord}.
\end{notation}
By definition, 
\begin{equation}\label{1_eq}
\{ P~|~\tr(P)=0 \} \supset \{ P~|~\x(P)=0 \}. 
\end{equation}    
\begin{definition}[$n$-gon, strong $2$-gon]
For a knot projection $P$, let $p$ be the boundary of a closure of a connected component of $S^2 \setminus P$.  Let $n$ be a positive integer.  We call $p$ an $n$-gon if, when the double points of $P$ that lie on $p$ are removed, the reminder consists of $n$ connected components, each of which is homeomorphic to an open interval.  
In particular, if a $2$-gon is globally connected, as shown by the dotted curves in Figure~\ref{f7}, it is called a \emph{strong}~$2$-gon.  
\begin{figure}[h!]
\includegraphics[width=3cm]{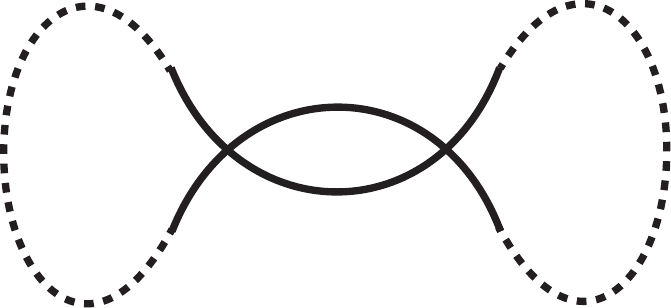}
\caption{Strong~$2$-gon.}\label{f7}
\end{figure}
\end{definition}
By definition, if a $2$-gon is the boundary $\partial D$ of a disk $D$,   no curves intersect $D \setminus \partial D$.  
\begin{notation}\label{not:1b2b}
The deformation on the left  Figure~\ref{f8} is denoted by $1b$, on the right by $s2b$.   
\end{notation}
\begin{figure}[h!]
\includegraphics[width=8cm]{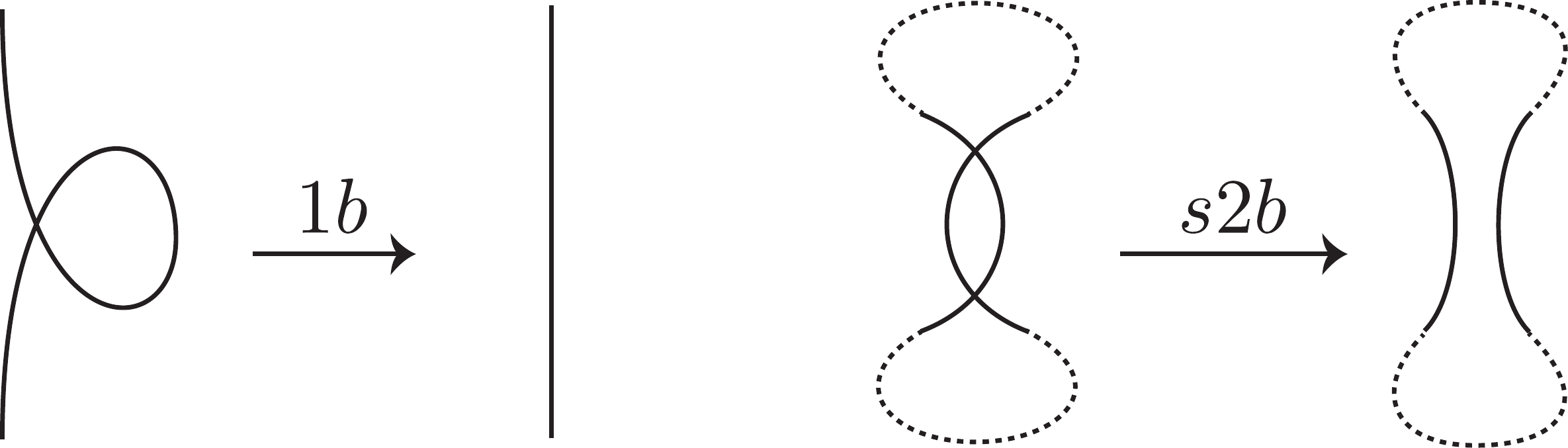}
\caption{$1b$ and $s2b$.  The dotted curves indicate connections.}\label{f8}
\end{figure}
\begin{theorem}\label{main_thm}
{\it
Let $P$ be a knot projection.  If $CD_P$ contains no triple chords, there exists a finite sequence consisting of $1b$'s and $s2b$'s, as in Notation~\ref{not:1b2b}, from $P$ to a simple closed curve.  
}
\end{theorem}
\begin{remark}
A special case of Theorem~\ref{main_thm} is given via a graph-theoretical machinery and checking 36 cases  \cite{ItoTakimura2016T}.  
\end{remark}
Here, we recall Fact~\ref{fact_circle}.  
\begin{fact}[\cite{ItoTakimura2016S}]\label{fact_circle}
Let $P$ be a knot projection.  The following two statements are mutually equivalent.  
\begin{enumerate}
\item $U$ is obtained from $P$ by a finite sequence consisting of $1b$'s, $s2b$'s, inverses of $1b$'s, and inverses of $s2b$'s.  
\item $U$ is obtained from $P$ by a finite sequence consisting of $1b$'s, $s2b$'s.  
\end{enumerate}
\end{fact}
\begin{definition}\label{notation3}
Let $\mathcal{S}$ $=$ 
$\{ P~|~$ $P$ is obtained from $U$ by a finite sequence consisting of inverses of $1b$'s or $s2b$'s, as shown in Figure~\ref{f8} $\}$.  
\end{definition}
By Notation~\ref{notation2} and Definition~\ref{notation3}, known results \cite{SakamotoTaniyama2013, ItoTakimura2016S} imply that  
\begin{align}\label{1_eq4}
\{P~|~\h(P)=0\} \cap \mathcal{S} 
= \{ P~|~ &P~{\textrm{is obtained by a finite sequence consisting}}\\ &{\textrm{of $1b$'s from $P$ to}}~U \}. \nonumber 
\end{align} 
%
\begin{definition}[\cite{Polyak1998}]\label{defJ+}
Let $a_2$ be the second coefficient of the Conway polynomial.  For a knot projection $P$ with exactly $n$ double points, we obtain the possible $2^n$ knots in $\mathbb{R}^3$ by over/under informations.  Then, let $a_2 (P)$ be the average $a_2 (K)$ where $K$ varies over all $2^n$ knots.   Then, let $(J^+ + 2 St)(P)$ $=$ $8 a_2 (P)$.         
\end{definition}
\begin{corollary}\label{main1}
{\it
Let $P$ be a knot projection.  Let $(J^+ + 2 St)(P)$ be as in Definition~\ref{defJ+}, $\x(P)$ and $\tr(P)$ as in Notation~\ref{notation2}, and $\mathcal{S}$ as in  Definition~\ref{notation3}.    
Then,  
\begin{equation}\label{1_eq2}
\{P~|~ (J^+ + 2 St) (P) = 0 \} \supset {\mathcal{S}} \supset \{ P~|~\tr(P)=0 \} \supset \{ P~|~\x(P)=0 \}.  
\end{equation} 
}
\end{corollary}
\begin{remark}
 (\ref{ArnoldEq})--(\ref{1_eq2}) come together in the same story.  
\end{remark}
\section{\bf Proof of Theorem~\ref{main_thm}}\label{sec3}
Let $P$ be a knot projection having at least one double point.   
Suppose that $CD_P$ contains no triple chords.  Then, it is sufficient to show that $P$ has at least one $1$-gon or one strong~$2$-gon.   This is because it is easy to find a finite sequence consisting of $1b$'s and $s2b$'s to obtain a simple closed curve from $P$ under the condition that the chord diagram of a knot projection contains no triple chords in each step applying $1b$ or $s2b$.    
Further, in the following, we will show that $P$ has at least one $1$-gon or one $2$-gon since it is easy to see that a $2$-gon must be a strong $2$-gon under the condition that $CD_P$ contains no triple chords.  


 If we trace out $P$, then there exists a double point $C$.  Then, a simple closed curve bounds a $2$-disk $\delta$ with the boundary on $S^2$ where the curve starts and ends at $C$ $($Figure~\ref{04a}$)$.  
The disk is called a \emph{teardrop disk}  
and the double point $C$ is called a \emph{teardrop origin}.  
If a teardrop disk does not contain any other  teardrop disk, the teardrop disk is called an \emph{innermost teardrop disk}.  
Let $\delta$ be an innermost teardrop disk with the  teardrop origin $C$.    
Note that this implies that every arc in $\delta \setminus \partial \delta$ is a simple arc.  

\begin{figure}[h!]
\includegraphics[width=2cm]{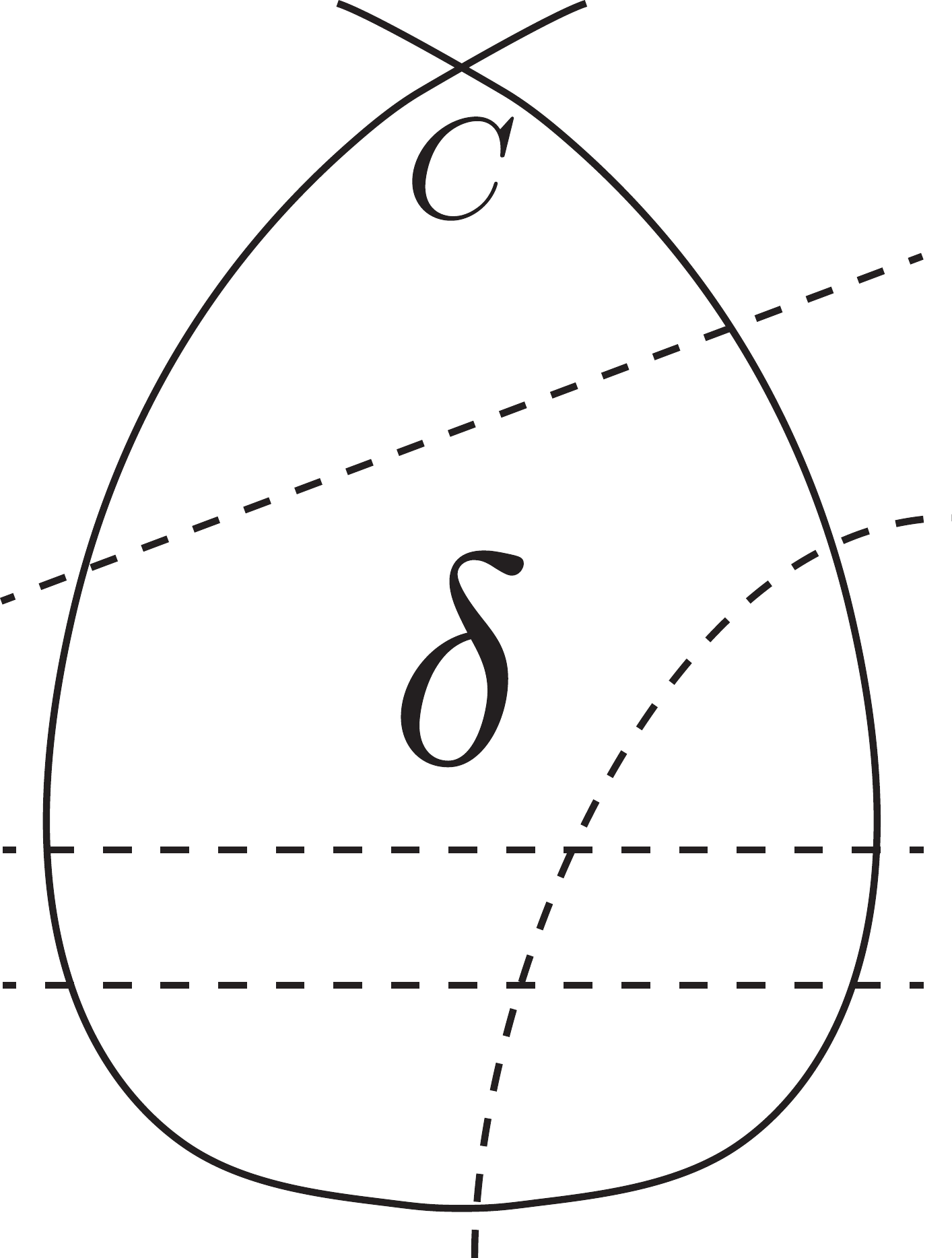}
\caption{A teardrop disk $\delta$, where the dotted curves indicate other  possible parts of a given knot projection $P$.}\label{04a}
\end{figure}

\noindent{\bf{Step~1}}.  We claim: If $CD_P$ contains no triple chords, then, for the teardrop disk $\delta$, there exist $2n+1$ double points on $\partial \delta$ and $\delta$ is of the form as shown in Figure~\ref{05}.  

It is easy to prove this claim  for the reader who is familiar with \cite{Taniyama1989}.  However, we will explain it as follows.    
\begin{figure}[h!]
\includegraphics[width=3cm]{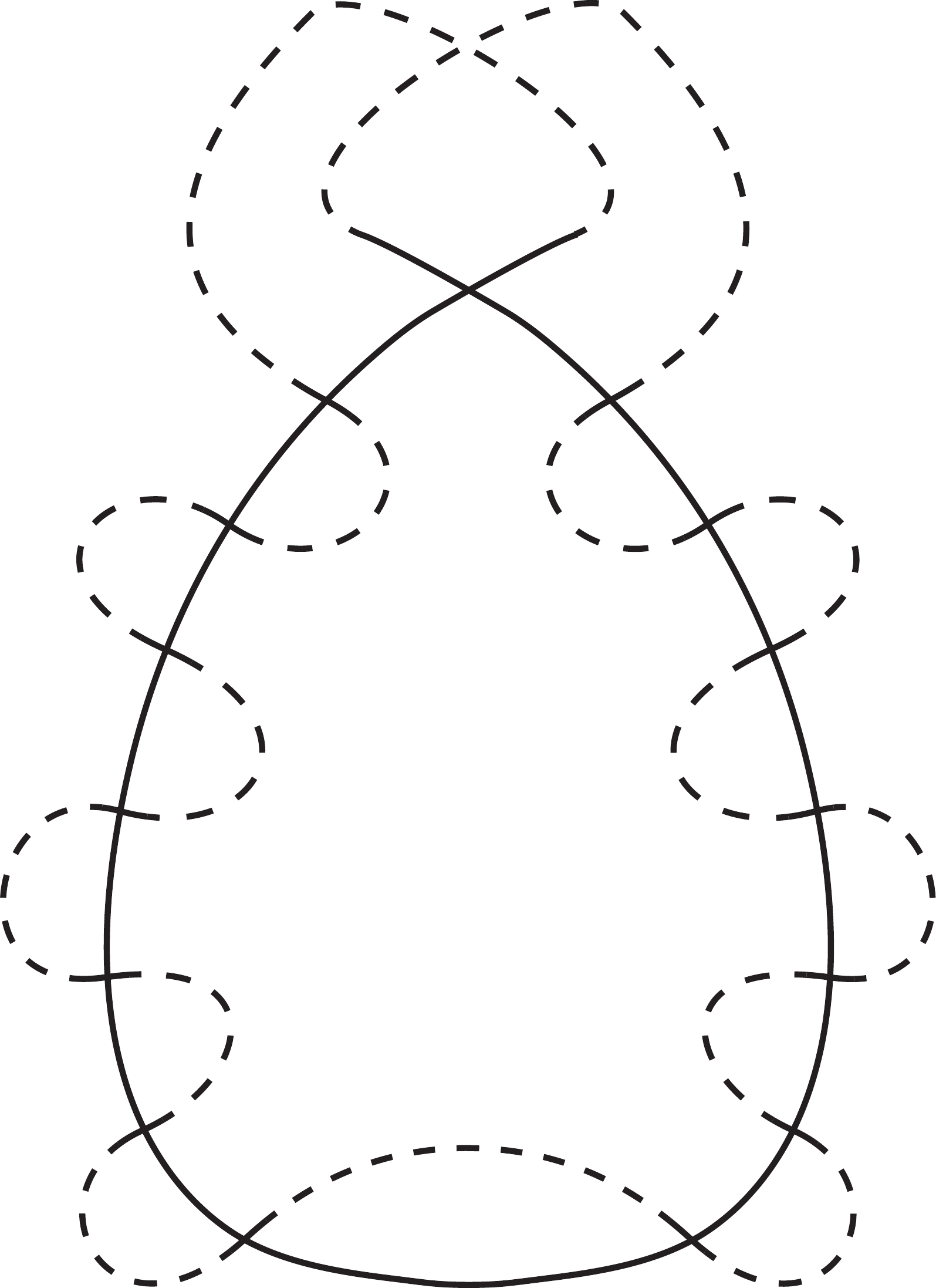}
\caption{Knot projection having $2n+1$ double points on $\partial \delta$ (this figure represents the case $n=7$).  The dotted parts indicate connections.}\label{05}
\end{figure}

We orient $P$ such that this induces a clockwise orientation on $\partial \delta$.  Using the orientation on $P$ we label the double points on $\partial \delta$ by $C$, $P_1$, $P_2$, \dots, $P_{2n}$ and on $P \setminus \partial \delta$ by $Q_1$, $Q_2$, \dots, $Q_{2n}$; see Figure~\ref{06}.  
Then, we obtain a permutation $\sigma$ of $(1, 2, \ldots, 2n)$ such that $Q_i$ $=$ $P_{\sigma(i)}$ for $i \in \{1, 2, \ldots, 2n\}$.  
For $1 \le i < j \le 2n$, we denote by $\overline{Q_i Q_j}$ the image of an sub-arc of $S^1$ which starts $Q_i$ and ends $Q_j$ with respect to the orientation of $P \setminus \partial \delta$ as the above and which has finite points of the intersection with $\partial \delta$.  See an example in Figure~\ref{06}.  

\begin{figure}[h!]
\includegraphics[width=7cm]{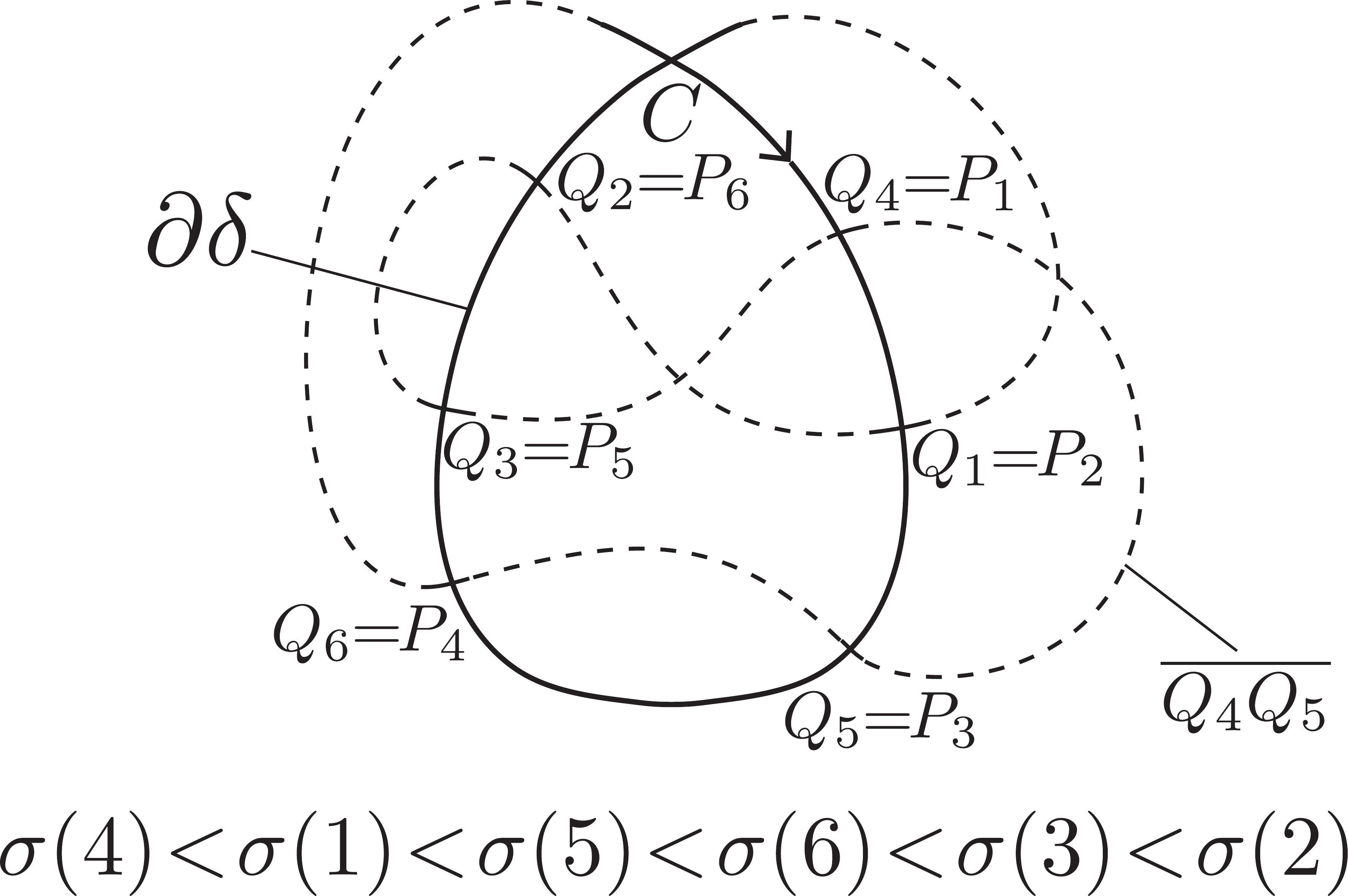}
\caption{An example of a knot projection $P$ with labels $P_i$ and $Q_i$ ($i=1, 2, \dots$).}\label{06}
\end{figure}

Under the above convention, for $P$ and for any $i$ and $j$ ($i < j$), if  $\sigma(i) < \sigma(j)$,  then $CD_P$ contains a triple chord (Figure~\ref{tear}).
\begin{figure}[h!]
\includegraphics[width=12cm]{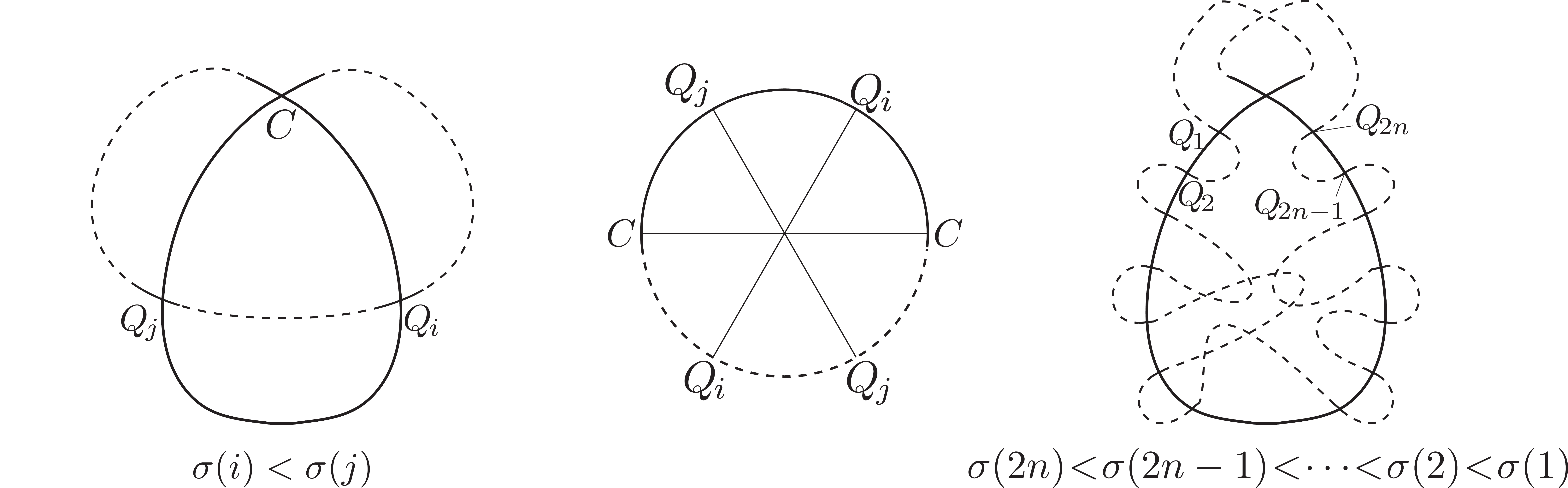}
\caption{The leftmost and the center figures indicate a knot projection and its chord diagram having {\it at least one} pair $i$ and $j$ satisfying $\sigma(i) < \sigma(j)$ ($i<j$).  If  every pair consisting of double points on $\partial \delta$ satisfies $\sigma(i) > \sigma(j)$ ($i<j$), the configuration of the double points on $\partial \delta$  is as shown in the rightmost figure.  In the rightmost figure, the dashed arcs can intersect each other.}\label{tear}
\end{figure}   
Thus, if $CD_P$ contains no triple chords (Figure~\ref{tear}), then $P$ is of the form as shown in the rightmost figure of Figure~\ref{tear}.    
\\


{\bf{Step~2}.} Let $\overline{Q_\lambda  Q_{\lambda+1}}$ ($1 \le \lambda < 2n$) be an arc in the teardrop disk $\delta$.  We note that this means that $\overline{Q_\lambda  Q_{\lambda + 1}}$ is a simple arc
\footnote{That is, for a knot projection $P$ and an generic immersion $f$ of $f(S^1)=P$, there exists an interval $I$ ($\subset S^1$) such that $f(I)$ $=$ $\overline{Q_\lambda  Q_{\lambda + 1}}$ and $f|_I$ is a bijection.} and    $\lambda$ is an odd integer.  
If $\overline{Q_k Q_{k+1} }$ does not intersect any other $\overline{Q_j Q_{j+1} }$ ($k \neq j$), then there exists the strong~$2$-gon whose boundary contains $\overline{Q_k Q_{k+1} }$ in $\delta$, as shown in Figure~\ref{case1} and we are done.  Thus we may suppose that any $\overline{Q_k Q_{k+1} }$ intersects at least one other arc $\overline{Q_j Q_{j+1} }$ ($k \neq j$).  This leaves us with the following cases:  

\begin{figure}[h!]
\includegraphics[width=3cm]{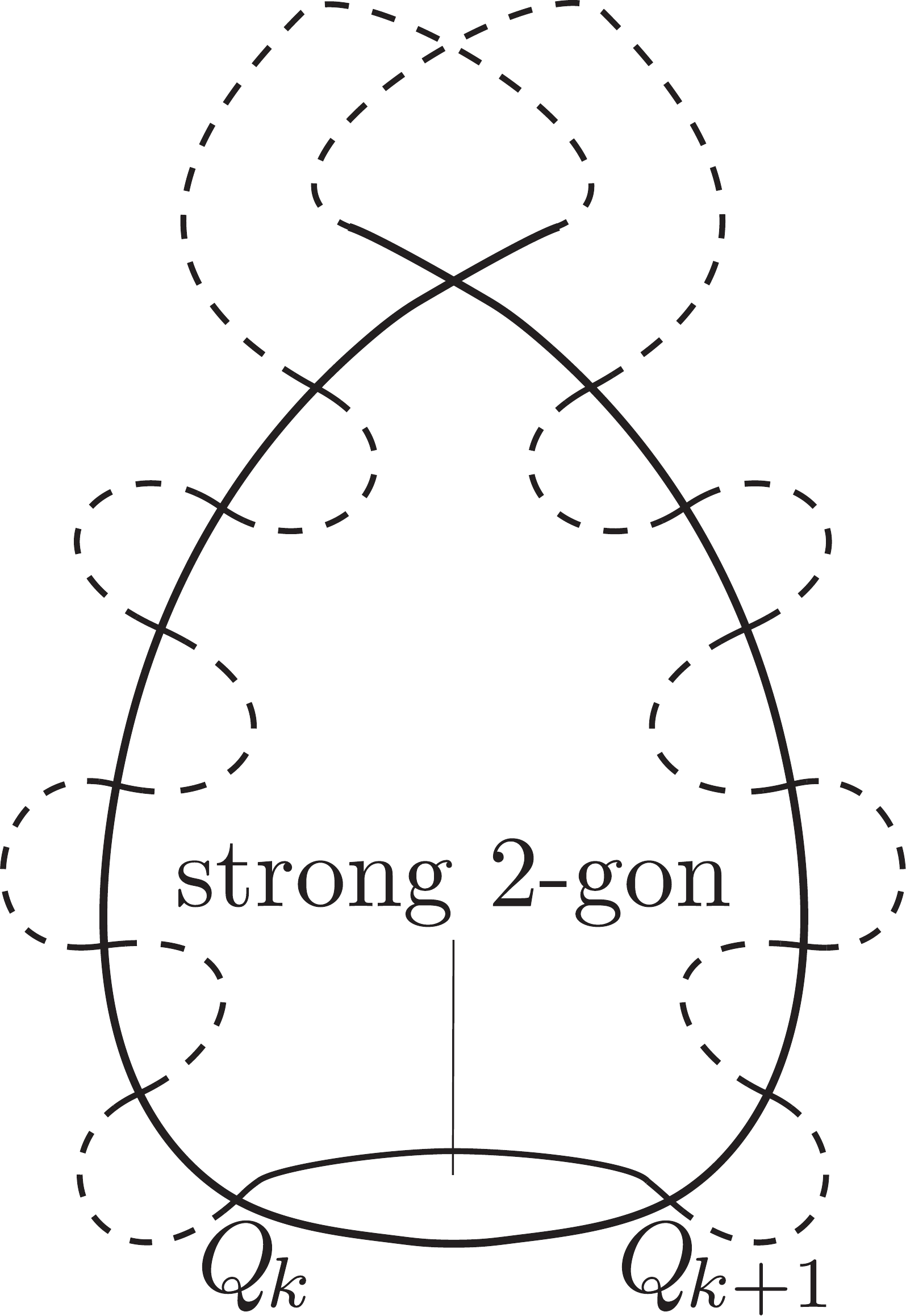}
\caption{A strong $2$-gon exists.}\label{case1}
\end{figure}

{\sf{Case~1}} (an argument on exactly two arcs): there exists $\overline{Q_j Q_{j+1} }$ such that it intersects another arc $\overline{Q_i Q_{i+1} }$ ($j \neq i$) and there is no such $\overline{Q_k Q_{k+1} }$ that intersects both $\overline{Q_j Q_{j+1} }$ and $\overline{Q_i Q_{i+1} }$ ($j \neq k$ and $i \neq k$).  

Without loss of generality, we may suppose $i < j$.  We orient the knot projection $P$ where $\partial \delta$ is oriented clockwise, and we trace $P$.  
Then the intersection $\overline{Q_j Q_{j+1}} \cap~\overline{Q_i Q_{i+1}}$ includes at least two double points, say, $x$ and $y$ in the order of  tracing $P$.    
In the other words, the part $\overline{Q_i Q_{i+1}}$ passes thorough double points in the order: $Q_i$, $x$, $y$, $Q_{i+1}$ (Figure~\ref{case2}).  
Then, if the intersections are ordered as in the center figure of Figure~\ref{case2} (i.e., the order is of type $Q_i$, $x$, $y$, $Q_{i+1}$, $x$, $y$), $CD_P$ contains the triple chord consisting of $x$, $y$ and $Q_{i+1}$.  
\begin{figure}[h!]
\includegraphics[width=10cm]{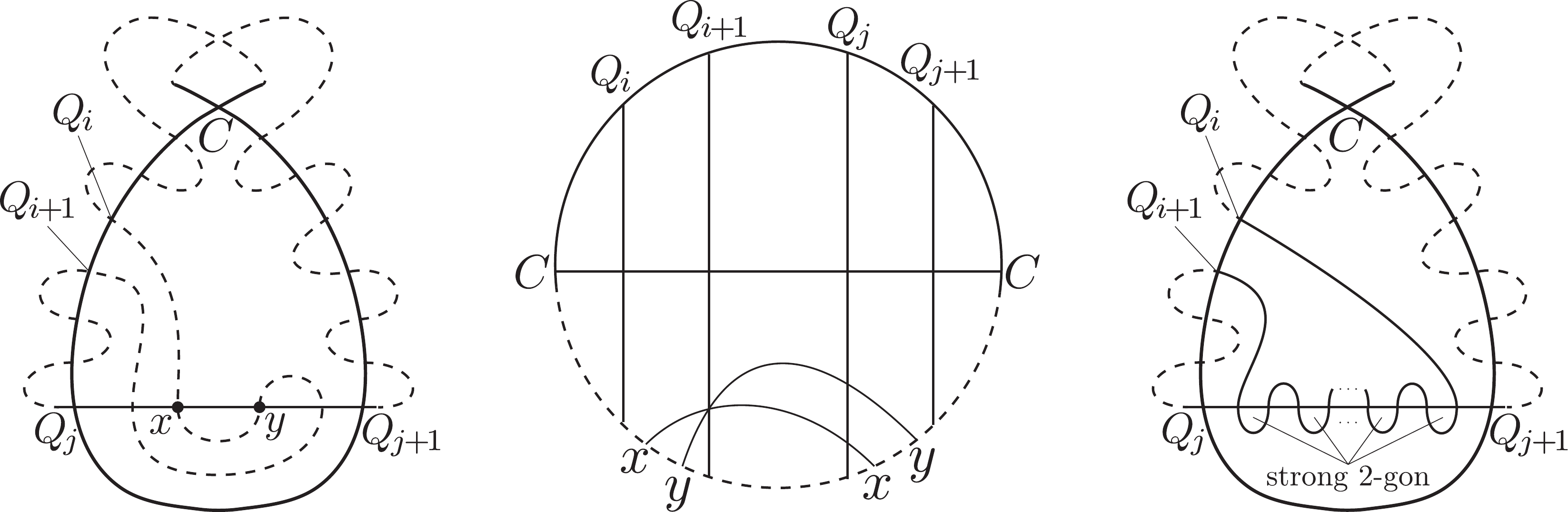}
\caption{Case~1.}\label{case2}
\end{figure}
Therefore, since every order of intersections is of type $Q_i$, $x$, $y$, $Q_{i+1}$, $y$, $x$, the intersection $\overline{Q_j Q_{j+1}} \cap \overline{Q_i Q_{i+1}}$ is as shown in the rightmost figure in Figure~\ref{case2} and $P$ has a $2$-gon.   

{\sf{Case~2}} (an argument on at least three arcs):
there exists $\overline{Q_j Q_{j+1} }$ such that it intersects another arc $\overline{Q_i Q_{i+1} }$ ($j \neq i$) and there is $\overline{Q_k Q_{k+1} }$ that intersects either $\overline{Q_j Q_{j+1} }$ or $\overline{Q_i Q_{i+1} }$ ($j \neq k$ and $i \neq k$).  
  
Note that, in Case~2, without loss of generality, we may suppose the following three conditions (e.g., see the leftmost figure of Figure~\ref{case3}).  
\begin{enumerate}
\item $1 \le i < j < k < 2n$. \\
(Every other case can be obtained by using mirror symmetries.)
\item There exists a lune region $R$ ($\subset \delta$) consisting of a part of $\delta$ and $\overline{Q_j Q_{j+1} }$.   \\
(If not, this case returns to the case as in Figure~\ref{case1}.)
\item $\overline{Q_i Q_{i+1} }$ and $\overline{Q_j Q_{j+1} }$ compose another lune region $r$ ($\subset R$).  \\
(If not, the case returns to Case~1.)
\end{enumerate}

\noindent $\ast$ Case~2-1: Suppose that $\overline{Q_k Q_{k+1}}$ does not intersect $r$.  In this case, there exists a $2$-gon as in the leftmost figure of Figure~\ref{case3}.    

\noindent $\ast$ Case~2-2: Suppose that $\overline{Q_k Q_{k+1}}$ intersects $r$.  In this case, if $\overline{Q_k Q_{k+1}}$ and $\overline{Q_j Q_{j+1}}$
 compose the innermost lune region in $r$,  this lune region gives a $2$-gon.   
If not, by contradiction, we will show that such a knot projection does not exist in the  following ($\flat$).

\noindent($\flat$) Assume that $r \setminus \partial r$ contains an arc that intersects $\partial r \setminus \overline{Q_j Q_{j+1}}$.  
Then, we orient the knot projection $P$ where $\partial \delta$ is oriented clockwise, and we trace $P$.   
When we start at $Q_{k}$, we go into $r$, and, by assumption, we intersect at $\partial r \setminus \overline{Q_j Q_{j+1}}$.  This intersection point is denoted by $x$.  
After we pass through $x$, we trace the curve, and we must intersect at $\overline{Q_{j} Q_{j+1}}$ because $r$ $\subset R$.  This second intersection is denoted by $y$.   Then, it is easy to see that $Q_k$, $x$, and $y$ on $CD_P$ compose a triple chord, as shown in Figure~\ref{case3}, which implies a contradiction.       
\label{1}

\begin{figure}[h!]
\includegraphics[width=10cm]{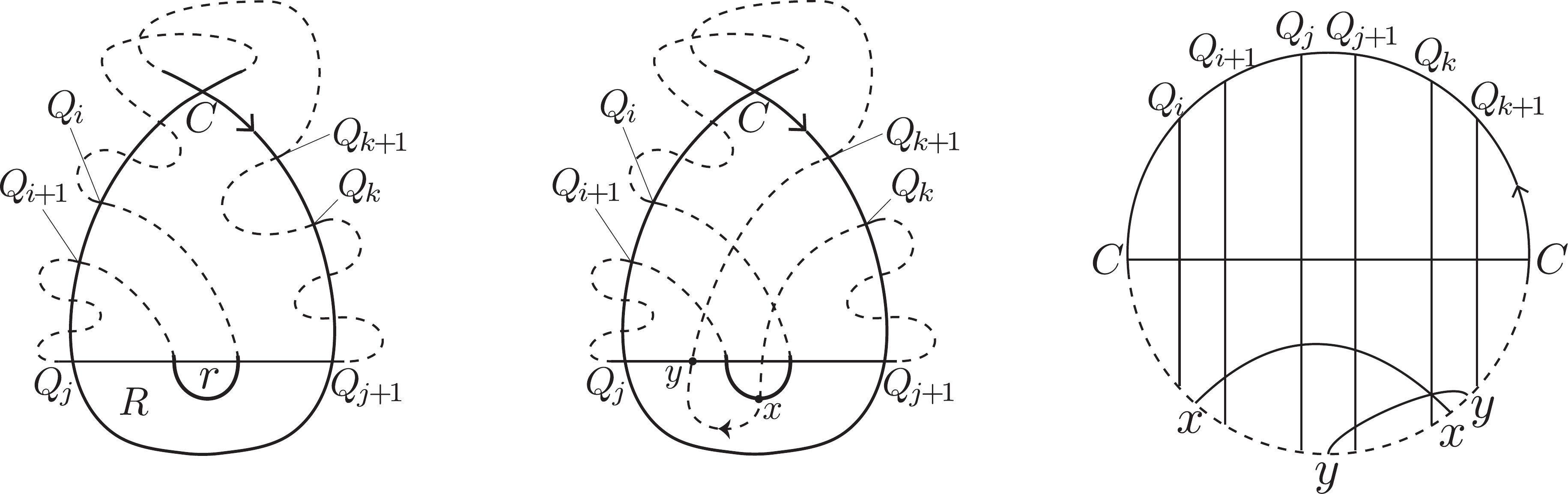}
\caption{Case~2.}\label{case3}
\end{figure}

Now Cases~1 and 2 have been checked.  
As a result, for $P$ which $CD_P$ contains no triple chord, if $n > 0$, $\delta$ is of the form satisfying Case~1 or Case~2, then $P$ has a strong~$2$-gon, and if $n=0$, $\partial \delta$ is a $1$-gon.  The proof of Theorem~\ref{main_thm} is complete.  
$\hfill\Box$

\begin{remark}
In Case~2, note that the definitions of $x$ and $y$ do not depend on their positions.   
In fact, there exist three possibilities (i), (ii), or (iii), as shown in Figure~\ref{f13}.  However, for every case, this does not depend on the definitions of intersection points $x$ and $y$, and thus, the same argument works.    
\begin{figure}[h!]
\includegraphics[width=8cm]{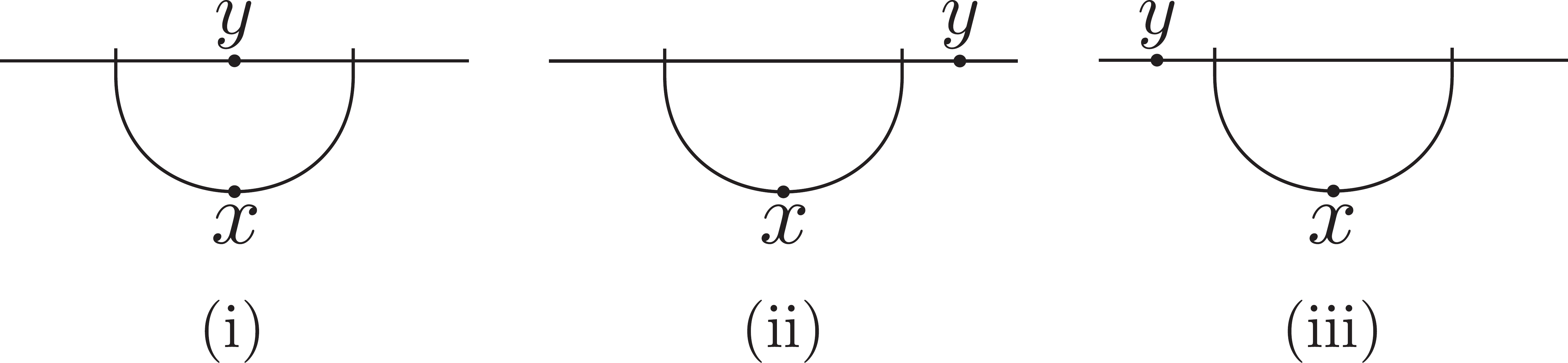}
\caption{Cases (i)--(iii).}\label{f13}
\end{figure}
\end{remark}

\section{\bf Application}\label{sec5}
A knot projection is \emph{reduced} if there is no double point whose deletion disconnects the knot projection.   
For a reduced knot projection, we have Corollary~\ref{cor2} from Theorem~\ref{main_thm}.  
\begin{corollary}\label{cor2}
{\it
Let $P$ be a reduced knot projection having at least one double point.  If $CD_P$ contains no triple chords, $P$ has at least two strong $2$-gons.    
}
\end{corollary}
\begin{remark}
Taniyama, in a personal communication, conjectured Corollary~\ref{cor2}, and advised us  that if Corollary~\ref{cor2} is true, Theorem~\ref{main_thm} holds.   
\end{remark}
Before starting the proof, 
we need to provide Definition~\ref{primeA},  
Lemma~\ref{lem_2}, and Lemma~\ref{lem_3}.  
\begin{definition}\label{primeA}
Let $P$ be a knot projection having at least one double point.  
Then, $P$ is a $4$-regular graph that is either $4$-edge connected or $2$-edge connected.  It is \emph{prime} if it is $4$-edge connected, otherwise it can be decomposed into pieces through the deletion of two edges.  
\end{definition}
For a knot projection $P$ having at least one double point, if $P$ is not prime, by Definition~\ref{primeA}, it is presented as a ``connected sum''  of two knot projections $P_1$ and $P_2$, each of which is not $U$ (for the definition of connected sums of knot projections, see~\cite{Ito2019}).  Then, $P$ is denoted by $P_1 \sharp P_2$.   

By the argument as above, we immediately have:  
\begin{lemma}\label{lem_2}
{\it 
Let $P_i$ be a knot projection $(i=1, 2)$.  
If each $CD_{P_i}$ contains no triple chord, then $CD_{P_1 \sharp P_2}$ contains no triple chord.  
}
\end{lemma}
\begin{lemma}\label{lem_3}
{\it
Let $P$ be a reduced knot projection having at least one double point.  If $P$ is prime and $CD_P$ contains no triple chords, then $P$ has at least two strong $2$-gons.   
}
\end{lemma}
\begin{proof}
Let $P$ be a reduced knot projection having at least one double point.  Note that $P$ has no $1$-gons since $P$ is a reduced knot projection.  
By the assumption, $P$ is prime and $CD_P$ contains no triple chords.  Thus, Theorem~\ref{main_thm} implies that $P$ has at least one strong $2$-gon.  

We also note that by Lemma~\ref{lem_2}, $CD_{P \sharp P}$ contains no triple chords.  Since $P$ is a reduced knot projection, $P \sharp P$ is a reduced knot projection, which implies that $P \sharp P$ has no $1$-gons.  Thus, Theorem~\ref{main_thm} implies that $P \sharp P$ has at least one strong $2$-gon.  

Now, assume that $P$ has only one strong $2$-gon.  Then, there exists a connected sum $P \sharp P$ that has no strong $2$-gon (Figure~\ref{19}), which contradicts the above fact that $P \sharp P$ has at least one strong $2$-gon.  
\begin{figure}[htbp]
\includegraphics[width=10cm]{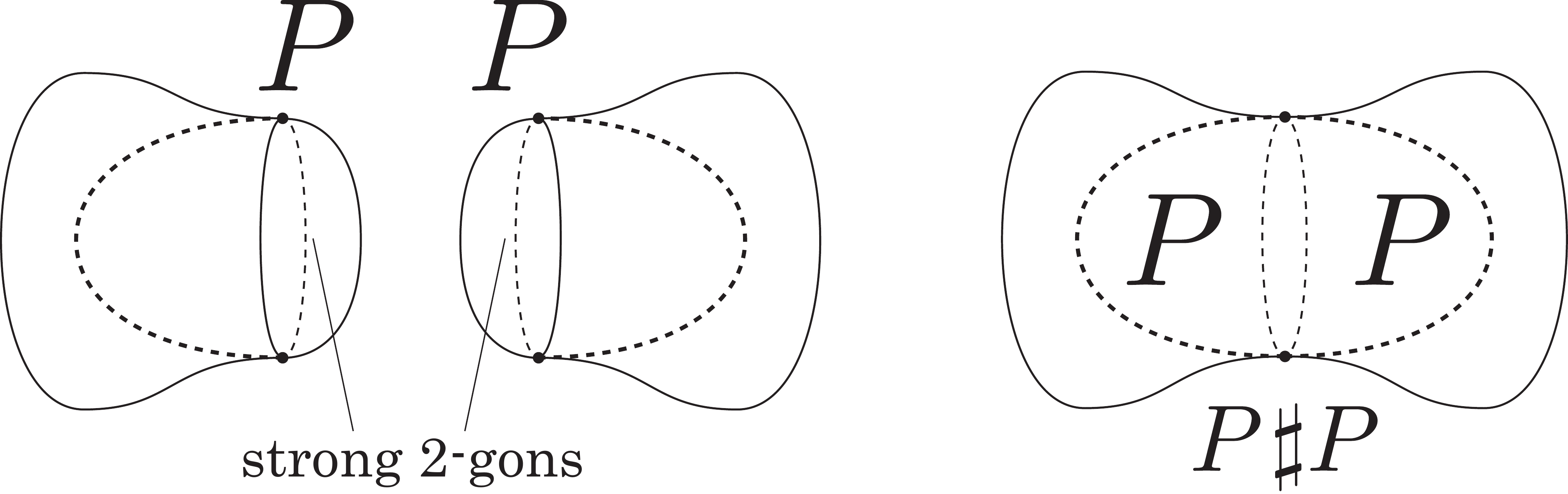}
\caption{The number of strong $2$-gons decreases by $2$.}\label{19}
\end{figure}
Therefore, $P$ has at least two strong $2$-gons.   
\end{proof}
Now, we prove Corollary~\ref{cor2}.  

\noindent{\it Proof of Corollary~\ref{cor2}.} 
It is known that for any knot projection, there exists a prime decomposition \cite{ItoTakimura2016S}, 
i.e., for any knot projection $P$, there exist a positive integer $l$ and the set of some knot projections $\{P_i\}_{i=1}^l$ such that
each $P_i$ is prime and
 \[P = (\cdots ((P_1 \sharp P_2) \sharp P_3) \cdots P_l).\]

Now, let $P'$ be a reduced knot projection having at least one double point where $CD_{P'}$ contains no triple chords.  Using the above fact, there exists a prime decomposition such that 
\begin{equation}\label{eq1}
 P' = (\cdots ((P'_1 \sharp P'_2) \sharp P'_3) \cdots P'_{l'}),  
\end{equation}
where each $P'_i$ is prime.  Note that each $P'_i$ is a reduced knot projection because $P$ is a reduced knot projection.      
Then, Lemma~\ref{lem_3} implies that each $P'_i$ has at least two strong $2$-gons.  

For any knot projection $P$, let $n_P$ be the number of strong $2$-gons.  
If a knot projection $Q_i$ satisfies $n_{Q_i} \ge 2$ ($i=1, 2$), then $n_{Q_1 \sharp Q_2} \ge 2$.   
Here, note that two strong $2$-gons cannot share an edge in a knot projection (it is easy to see the fact by using Figure~\ref{f7}).  
Therefore, (\ref{eq1}) implies that $n_{P'} \ge 2$.  Therefore, $P'$ has at least two strong $2$-gons.  
$\hfill\Box$

\section*{\bf Acknowledgement}
The authors would like to thank the referee for the comments. 

\bibliographystyle{plain}
\bibliography{Ref}
\end{document}